\documentclass[reqno,11pt]{amsart}
\usepackage{amsfonts}
\usepackage{amsmath}
\usepackage{amsthm}
\usepackage{amssymb}
\usepackage{mathtools}
\usepackage{mathrsfs}
\usepackage{bbding}
\usepackage{bbm}
\usepackage{upgreek}
\usepackage[latin1,utf8]{inputenc}
\usepackage{libertine}
\usepackage{dsfont}
\usepackage{fancyhdr}
\usepackage{enumerate}
\usepackage[hang,flushmargin]{footmisc} 

\usepackage{cancel}
\usepackage{graphicx}
\usepackage{float}
\usepackage[dvipsnames]{xcolor}

\usepackage[numbers]{natbib}

\usepackage{hyperref}
\hypersetup{
    colorlinks=true,
    linkcolor=Violet,
    citecolor=Violet,
    urlcolor=Violet,
}

\theoremstyle{plain}
\newtheorem{theorem}{Theorem}[section]
\newtheorem{lemma}{Lemma}[section]
\newtheorem{proposition}{Proposition}[section]
\newtheorem{proposition*}{Proposition}

\theoremstyle{definition}

\newtheorem{example}{Example}

\theoremstyle{remark}
\newtheorem{remark}{Remark}[section]

\numberwithin{equation}{section}

\makeatletter
\newcommand{\ltext}[2]{%
  \@bsphack
  \csname phantomsection\endcsname 
  \def\@currentlabel{#1}{\label{#2}}%
  \@esphack
}
\makeatother

\DeclareMathOperator{\N}{\mathbb{N}}
\DeclareMathOperator{\Z}{\mathbb{Z}}
\DeclareMathOperator{\R}{\mathbb{R}}
\DeclareMathOperator{\radius}{\mathnormal{r}}
\DeclareMathOperator*{\argmin}{arg\,min}
\DeclareMathOperator{\G}{\mathcal{G}}
\DeclareMathOperator{\Hh}{\mathcal{H}}
\DeclareMathOperator{\Poi}{\mathcal{P}}
\DeclareMathOperator{\tht}{\upvartheta}
\DeclareMathOperator{\Om}{\Upomega}
\DeclareMathOperator{\F}{\mathscr{F}}
\DeclareMathOperator{\A}{\mathscr{A}}
\DeclareMathOperator{\p}{\mathbb{P}}

\DeclareMathOperator{\dist}{\mathrm{D}}
\DeclareMathOperator{\E}{\mathbb{E}}

\title[Limiting shape for FPP models on RGGs]{Limiting shape for first-passage percolation models on random geometric graphs}

\author[C. F. Coletti \and
L. R. de Lima \and
A. Hinsen \and
B. Jahnel \and
D. Valesin]
{Cristian F. Coletti \and
Lucas R. de Lima \and
Alexander Hinsen \and
Benedikt Jahnel \and
Daniel Valesin}

\thanks{{\bf Funding:} This work was funded by the German Research Foundation under Germany's Excellence Strategy MATH+: The Berlin Mathematics Research Center, EXC-2046/1 project ID: 390685689, Orange Labs S.A., and the German Leibniz Association via the Leibniz Competition 2020. Research also supported by grants \#2017/10555-0, \#2019/19056-2, and \#2020/12868-9, S\~ao Paulo Research Foundation (FAPESP)}

\address{Center for Mathematics, Computation, and Cognition, Federal University of ABC\\
Av. dos Estados, 5001\\
09210-580 Santo Andr\'e, S\~ao Paulo\\
Brazil}
\email{cristian.coletti@ufabc.edu.br}
\email{lucas.roberto@ufabc.edu.br}

\address{Weierstrass Institute for Applied Analysis and Stochastics\\
Mohrenstraße 39\\
10117 Berlin\\
Germany}
\email{alexander.hinsen@wias-berlin.de}

\address{Institut f\"ur Mathematische Stochastik, Technische Universit\"at Braunschweig, Universit\"atsplatz 2,
38106 Braunschweig, Germany \& Weierstrass Institute for Applied Analysis and Stochastics\\
Mohrenstraße 39\\
10117 Berlin\\
Germany}
\email{benedikt.jahnel@tu-braunschweig.de}

\address{Department of Statistics, University of Warwick\\
Coventry, CV4 7AL\\
United Kingdom}
\email{daniel.valesin@warwick.ac.uk}

\keywords{Poisson-Gilbert graph; Richardson model; Bernoulli bond percolation; high-density limit; isotropic shapes}

\subjclass[2020]{Primary: 52A22, 60F15; Secondary: 60K35}

\begin{document}


\begin{abstract}
    Let a random geometric graph be defined in the supercritical regime for the existence of a unique infinite connected component in Euclidean space. Consider the first-passage percolation model with independent and identically distributed random variables on the random infinite connected component. We provide sufficient conditions for the existence of the asymptotic shape and we show that the shape is an Euclidean ball. We give some examples exhibiting the result for Bernoulli percolation and the Richardson model. 
    For the Richardson model we further show that it converges weakly to a nonstandard branching process in the joint limit of large intensities and slow passage times.
\end{abstract}

\maketitle


\section{Introduction, main results and examples}

First-passage percolation (FPP) was initially introduced by \citet{hammersley1965} to study the spread of fluids through random medium.  Since then, several variations of the percolation process have been extensively investigated (see \citet{auffinger2017} for an overview of FPP on $\Z^d$) due to their considerable amount of theoretical consequences and applications. It determines a random metric space by assigning random weights to the edges of a graph.

We consider the FPP model defined on a random geometric graph (RGG) in $\R^d$ with $d \geq 2$. Here, the RGG is defined as in \citet{penrose2003} by setting the vertices to be given by a homogeneous Poisson point process (PPP) with intensity $\lambda>0$ and the edges are defined between any pair of vertices that are within an Euclidean distance smaller than a fixed threshold $\radius>0$. This random graph is also known as the Poisson--Gilbert disk model. It is a graph associated to the Poisson--Boolean model in continuum percolation, and it can also be seen as a particular case of the random-connection model (see for instance \citet{meester1996}).

The properties and other details regarding the structure and definition of the process will be given later in the text.  We present here the basic definition in general lines. Let the parameters $(\lambda,r)$ of the RGG be supercritical for the almost-sure existence of a random infinite connected component $\Hh$. Note that the infinite component $\Hh$ is unique almost surely and we define the FPP model on $\Hh$ with independent and identically distributed random variables on the joint probability space $(\Om, \A,\p)$.

The aim of this paper is to investigate the $\p$-\textit{a.s.} existence of the limiting shape of the above defined process. In fact, we show that, under some conditions, the random balls of $\Hh$ converge $\p$-\textit{a.s.} to the deterministic shape of an Euclidean ball. The additional conditions refer to the distribution of zero passage time on the edges and the at-least linear growth of the process.

The model will be formally defined in the next sections, we give below a simplified description of the process to state the main result. Let $\tau$ be a random variable which defines the common distribution of the i.i.d.~passage times $\tau_e$ along each edge $e \in E(\Hh)$.

Set $\radius_c(\lambda)>0$ to be the critical $\radius$ for the existence of the infinite connected component $\Hh$ of the RGG $\G_{\lambda,r}$. Let $B_s(x)$ stand for the open Euclidean ball of radius $s \geq 0$ centered at $x \in \R^d$ and denote by $\upupsilon_d$ be volume of the unit ball in $d$-dimensional Euclidean space. Denote by $H_t$ the random subset of $\R^d$ of points for which their closest point in $\Hh$ is reached by the FPP model up to time $t>0$. We let $H_0$ be the set of points that have the same closest point in $\Hh$ as the origin. Here is our first main theorem.

\begin{theorem}[Shape theorem for FPP on RGGs] \label{thm:shape.FPP}
    Let $d \geq 2$ and $\radius>\radius_c(\lambda)$. Consider the FPP with i.i.d.~random variables
    defined on the infinite connected component $\Hh$ of  $\G_{\lambda,r}$. Suppose that the following conditions are satisfied:
    \begin{enumerate}[(${A}_1$)]
    \item \ltext{${A}_1$}{A1}
    We have that
    \[
        \p(\tau=0) < \frac{1}{\upupsilon_d \radius^d \lambda}.
    \]
    
    \item \ltext{${A}_2$}{A2}
    There exists $\eta>2d+2$ such that
    \[
        \E[\tau^\eta] < +\infty.
    \]
\end{enumerate}
Then, there exists $\upvarphi \in (0, +\infty)$ such that, for all $\varepsilon \in (0,1)$, one has $\p$-\textit{a.s.} that
    \begin{equation} \label{eq:asymptotic.cone}
        (1-\varepsilon) B_{\upvarphi}(o) \subseteq \frac{1}{n}H_{n} \subseteq (1+\varepsilon) B_{\upvarphi}(o)
    \end{equation}
    for sufficiently large $n \in \N$.
\end{theorem}

The existence of the limiting shape is particularly interesting because the RGG is a random graph which exhibits unbounded holes and unbounded degrees. To avoid the possible extreme effects of such pathologies on the growth of the process, we control the growth almost surely by combining the conditions above with properties of the point process.

The interest of applications for this class of models has already been pointed out by \citet{jahnel2020}. In particular, they suggested the theorem for the Richardson model on telecommunication networks. The example is naturally associated with the contact process by stochastic domination as studied by \citet{menard2016}, and \citet{riblet2019}. Another interesting application is a lower bound for the critical probability of bond percolation on the RGG. The same lower bound can be obtained by other methods (\textit{e.g.} branching processes), however it shows in comparison how good and suitable condition \eqref{A1} is.

It is worth pointing out that one can find in the literature a bigger class of random geometric graphs studied by \citet{hirsch2015} where the graph distance was interpreted as a FPP model. It suggests that the class of RGGs could also be expanded in our case. We chose to focus our attention on the standard definition in this work due to the usage of intermediate results presented in the next section.

Before we state our second main result, let us present some examples.

\begin{example}[Bond percolation]
    We define the bond percolation by considering the clusters of the Bernoulli FPP only at time zero, see Figure~\ref{fig:bond_perc_RGG} for an illustration. For this, let us call $e \in E(\Hh)$ an \emph{open edge} when $\tau_e =0$. Set $\tau_e \sim \operatorname{Ber}(1-p)$ independently for every $e \in E(\Hh)$ and observe that \eqref{A2} is immediately satisfied.
    
    Then, the open clusters are maximally connected components defined by sites with passage time zero between them. Let us define the critical probability $p_c$ for the bond percolation on the $d$-dimensional RGG by
    \[
        p_c := \inf\{p \in [0,1]: \p(\exists \text{ an infinite open cluster in } \Hh)>0,\tau_e \sim \operatorname{Ber}(1-p)\}.
    \]
    \begin{figure}[!ht]
        \centering
        \fbox{\includegraphics[width=325pt]{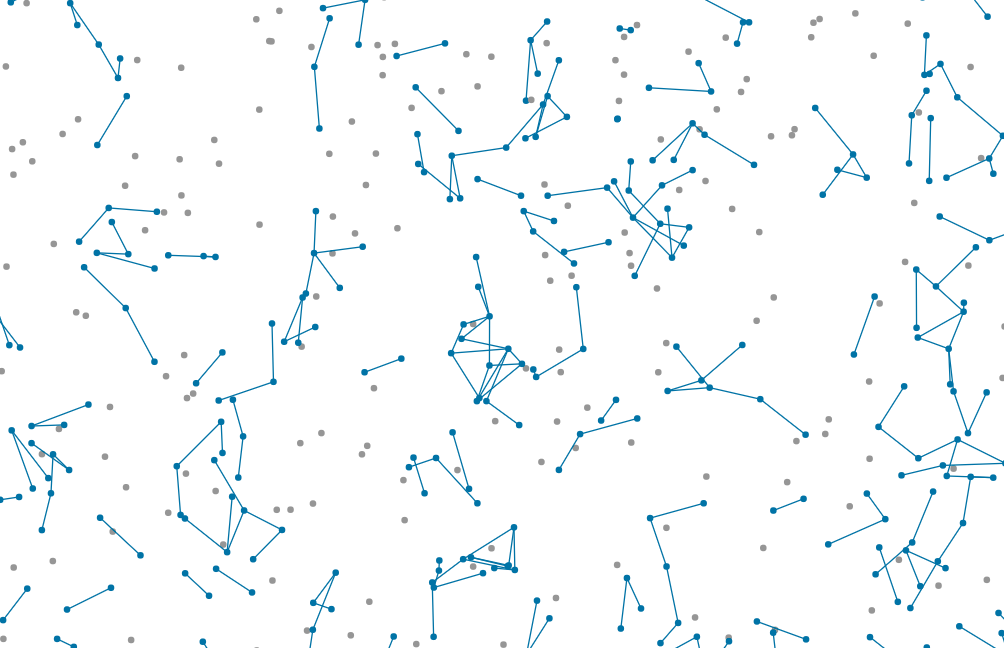}}
        \caption{Simulation of the open clusters for a bond percolation model on a 2-dimensional RGG with $p<1/(\upupsilon_d\radius^d\lambda)$.}
        \label{fig:bond_perc_RGG}
    \end{figure}
    Note that by Theorem \ref{thm:shape.FPP} the case $p< 1/(\upupsilon_d \radius^d \lambda)$ implies the existence of the limiting shape. Thus, an immediate consequence of the theorem is the following lower bound for the critical probability
    \[
        p_c \geq 1/(\upupsilon_d \radius^d\lambda),
    \]
    and for $p=0$ we recover $\Hh$. 
    We observe that the same lower bound can also be obtained by exploration methods.
\end{example}

\begin{example}[Richardson's growth model] 
    Consider the interacting particle system known as the Richardson model defined on the infinite connected component $\Hh$ of the RGG with parameter $\lambda_{\rm I}>0$. It is a random growth process based on a model introduced by \citet{richardson1973} and illustrated in Figure~\ref{fig:Richardson_RGG}. It is commonly referred to as a model for the spread of an infection or for the growth of a population.
    
    At each time $t \geq 0$, a site of $\Hh$ is in either of two states, healthy (vacant) or infected (occupied). Let $\zeta_t:V(\Hh)\to \{0,1\}$ indicate the state of the sites at time $t$ assigning the values $0$ and $1$ for the healthy and infected states, respectively. The process evolves as follows:
    \begin{itemize}
        \item A healthy particle becomes infected at rate $\lambda_{\rm I}\sum_{y \sim x}\zeta_t(y)$ and
        \item an infected particle remains infected forever.
    \end{itemize}

    It is easily seen that the process is determined by FPP with edge passage times $\tau_e \sim \mathrm{Exp}(\lambda_{\rm I})$ independently for each $e \in E(\Hh)$. In particular, this version of the Richardson model conventionally stochastically dominates the basic contact process.
    \begin{figure}[!ht]
        \centering
        \fbox{\includegraphics[scale=0.27]{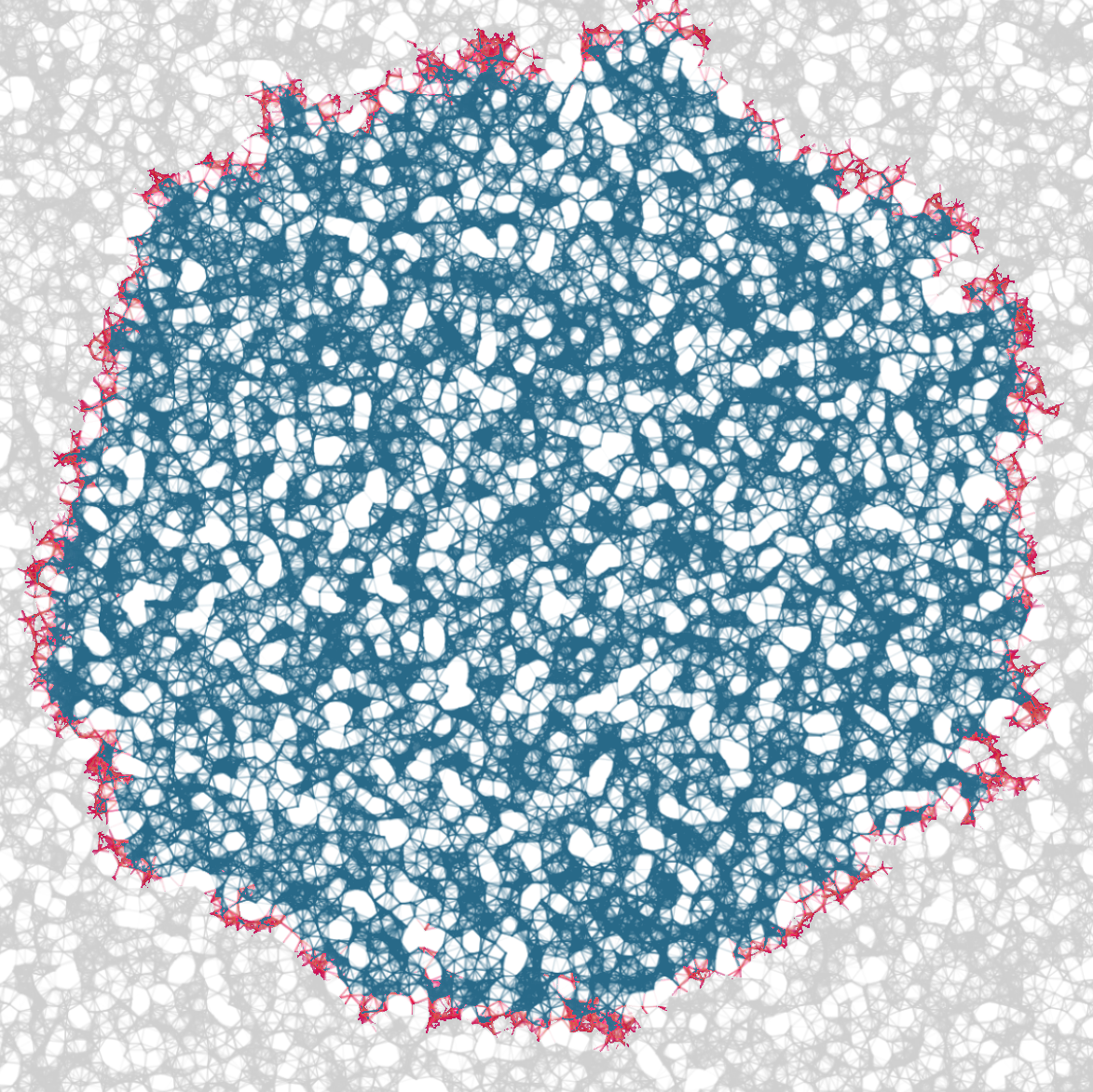}}
        \caption{Simulation of the spread of an infection given by the Richardson model on a bidimensional RGG.}
        \label{fig:Richardson_RGG}
    \end{figure}

    Conditions \eqref{A1} and \eqref{A2} are straightforward since $\p(\tau=0)=0<1/(\upupsilon_d \radius^d\lambda)$ and since  $\E[\exp(\alpha\tau)] < +\infty$ for $\alpha \in (0, \lambda_{\rm I})$. Hence, Theorem~\ref{thm:shape.FPP} is valid for the Richardson model on $\Hh$ for any supercritical $\radius> \radius_c(\lambda)$.
    
    Futhermore, it is immediate to see that Theorem \ref{thm:shape.FPP} still holds for any initial configuration $\mathrm{Z} \subseteq \R^d$ of infected particles whenever $\mathrm{Z} \subseteq B_{s'}(o)$ for some $s'>0$. In that case, we simply replace $H_t$ by $H_t^\mathrm{Z} := \bigcup_{z \in \mathrm{Z}}H_t^z$.
\end{example}

Our second main result concerns the asymptotic behavior of the Richardson model in the limit as $\alpha$ diverges to infinity in $\alpha\lambda$ and $\lambda_{\rm I}/\alpha$. In words, we consider a coupled limit of high densities and slow infection rates.

Indicating the parameters now explicitly, we write $\mathcal{H}_t^{\lambda,\lambda_\mathsf{I}}$ for the set of points in $\mathcal{H}$ reached by the FPP model up to time $t > 0$. As before, $\Hh_{0}^{\lambda, \lambda_{\rm I}}$ is the closest point in $\Hh$ to the origin, with $\Hh_0^{\lambda, \lambda_{\rm I}} = \emptyset$ if there is no infinite component. 

The limiting process is a branching process $(\mathcal{T}^{\lambda,\lambda_{\rm I}}_t)_{t\ge 0}$ defined as follows. At time zero, the process has a node only at the origin, i.e., $\mathcal{T}^{\lambda,\lambda_{\rm I}}_0=o$. Then, iteratively, each node $X_i\in \R^d$ of the process produces offsprings independently according to a Poisson process in time with intensity $\upupsilon_d r^d\lambda\lambda_{\rm I}$ and the offsprings are placed independently and uniformly within $B_r(X_i)$.   Let us note that this process has similarities with the growth process as presented in~\citet{deijfen2003asymptotic}.  Our second main result is given below.

\begin{theorem}[Time-space rescaling for Richardson models] \label{thm:rescaling.FPP}
Let $d\ge 2$. For the Richardson model with parameters $\radius, \lambda, \lambda_{\rm I}$ where $\radius>\radius_{\rm c}(\lambda)$, we have that
\begin{align*}
(\Hh_{t}^{\alpha \lambda, \lambda_{\rm I}/\alpha})_{t \ge 0} \longrightarrow (\mathcal{T}^{\lambda, \lambda_{\rm I}}_{t})_{t \ge 0},
\end{align*}
weakly with respect to the Skorokhod topology based on the weak topology, as $\alpha$ tends to infinity.
\end{theorem}

More details regarding the topologies involved in the above convergence will be given in Section~\ref{sec:Scale}.

The rest of the manuscript is organized as follows. In Section~\ref{sec:RGG}, we have compiled some basic facts about the RGG and show results on the asymptotic behaviour of the infinite component $\Hh$. The FPP model is defined in detail in Section~\ref{sec:FPP} where we also present the proof of Theorem~\ref{thm:shape.FPP}. Finally, in Section~\ref{sec:Scale} we present the proof of Theorem~\ref{thm:rescaling.FPP}.

\section{On the Random Geometric Graph} \label{sec:RGG}

In this section we present the definition and parameters for the RGG and the existence of the infinite connected component. We also show some results about its geometry in order to study the asymptotic shape in the next section.

Let $\Poi_\lambda$ be the random set of points determined by the homogeneous PPP on $\R^d$ with intensity $\lambda>0$. The RGG $\G_{\lambda,r}=(V,E)$ on $\R^d$ is defined by
\[
  V = \Poi_\lambda \quad \text{and} \quad
  E = \big\{\{u,v\} \subseteq V: \|u-v\|<\radius,~ u \neq v\big\},
\]
where $\|\cdot\|$ is the Euclidean norm. Since $\lambda^{-1/d}\Poi_\lambda \sim \Poi_1$, we may regard $\lambda$ as fixed due to the homogeneity of the norm. We write $\G_{\radius} := \G_{1,\radius}$ and $\Poi=\Poi_1$. Set $\left(\Upxi,\F,\mu\right)$ to be the probability space induced by the construction of $\Poi$. Let us now introduce the group action $\tht:\R^d \curvearrowright\Upxi$ which is determined by the spatial translation as a shift operator. That is, $\Poi\circ\tht_z = \big\{v-z :v \in \Poi \big\}$. The following lemma is a classical result on PPPs, which can be found for example in \citet[Proposition 2.6]{meester1996}.

\begin{lemma} \label{lm:PPP.mixing}
    The homogeneous PPP is mixing on $(\Upxi,\F,\mu,\tht)$.
\end{lemma}

\begin{remark} \label{rmk:isometry}
    Let $S:\R^d\to\R^d$ be an isometry. Then, it is known that $S$ induces a $\mu$-preserving ergodic function $\sigma: \Upxi \to \Upxi$ where $S[\Poi] = \Poi \circ \sigma$.
\end{remark}

We are interested in studying the spread of an infection on an infinite connected component of $\G_{\radius}$. It is a well-known fact from continuum percolation theory (see \citet{meester1996} or \citet[Chapter~10]{penrose2003} for details) that, for all $d \geq 2$, there exists a critical $\radius_c>0$ such that $\G_{\radius}$ has an infinite component $\Hh$ $\mu$-\textit{a.s.} for all $\radius>\radius_c$. Moreover, $\Hh$ is $\mu$-\textit{a.s.} unique. Since $\Hh$ is a subgraph of $\G_{\radius}$, we denote by $V(\Hh)$ and $E(\Hh)$ its sets of vertices and edges, respectively.

From now on, write $(\Upxi',\F',\mu)$ for the probability space of the PPP conditioned on the existence of $\Hh$ when $\radius>\radius_c$. { It suffices for our purposes to know that $\radius_c\geq 1/{\upupsilon_d}^{1/d}$ where $\upupsilon_d$ denotes the volume of the unit ball in the $d$-dimensional Euclidean space. Indeed, improved lower and upper bounds can be found in \citet{torquato2012} and $\radius_c$ approximates to $1/{\upupsilon_d}^{1/d}$ from above as $d \to +\infty$.}

Let us write $\theta_{\radius} := \mu \big( B_{\radius}(o) \cap V(\Hh) \neq \emptyset\big)$ and denote the cardinality of a set by $|\cdot|$. 
\begin{proposition}\label{prop:Hn.growth}\hspace{-0.1cm}{\rm{(Weaker version of~\citet[Theorem 1]{penrose1996}).}}
    Let $d \geq 2$, $\radius>\radius_c$ and $\varepsilon \in (0,1/2)$. Then, there exists $c>0$ and $s_0>0$ such that, for all $s \geq s_0$,
    \[
        \mu\left((1-\varepsilon)\theta_{\radius} <\frac{|V(\Hh) \cap [-s/2,s/2]^d |}{s^d}  <(1+\varepsilon)\theta_{\radius} \right) \geq 1 -\exp(-cs^{d-1}).
    \]
\end{proposition}

As a consequence of the last result, we present the following lemma without proof (see \citet[Lemma~3.3]{yao2011}).

\begin{lemma} \label{lm:ball.intersecting.infinite.comp}
    Let $\radius>\radius_c$. Then, there exists $C,C'>0$ such that, for each $x \in \R^d$ and all $s>0$,
    \[\mu\big(B_s(x) \cap V(\Hh)= \emptyset\big) \leq C\exp(-C's^{d-1}).\]
\end{lemma}

Let $\mathscr{P}(x,y)$ denote the set of self-avoiding paths from $x$ to $y$ in $\Hh$. The simple length of a path $\gamma = (x=x_0, x_1, \dots, x_m=y) \in \mathscr{P}(x,y)$ is denoted by $|\gamma|=m$.

Write $\dist(x,y)$ for the $\Hh$-distance between $x,y \in V(\Hh)$ given by
\[
    \dist(x,y) = \inf\{|\gamma|:\gamma \in \mathscr{P}(x,y)\}.
\]

Let $x \in \R^d$, then we define $q: \R^d \to V(\Hh)$ by
\begin{equation} \label{def:q.function}
    q(x) := \argmin_{y \in V(\Hh)}\{\|y-x\|\}
\end{equation}
the closest point to $x$ in the infinity cluster. 
Observe from \eqref{def:q.function} that $q$ may be multivalued for some $x \in \R^d$. In that case, we assume that $q(x)$ is uniquely defined by an arbitrarily fixed outcome of \eqref{def:q.function}. Hence $q$ induces a Voronoi partition of $\R^d$ with respect to $\Hh$, see Figure~\ref{fig:RGG_Voronoi} for an illustration.

\begin{figure}[!ht]
    \centering
    \fbox{\includegraphics[scale=0.245]{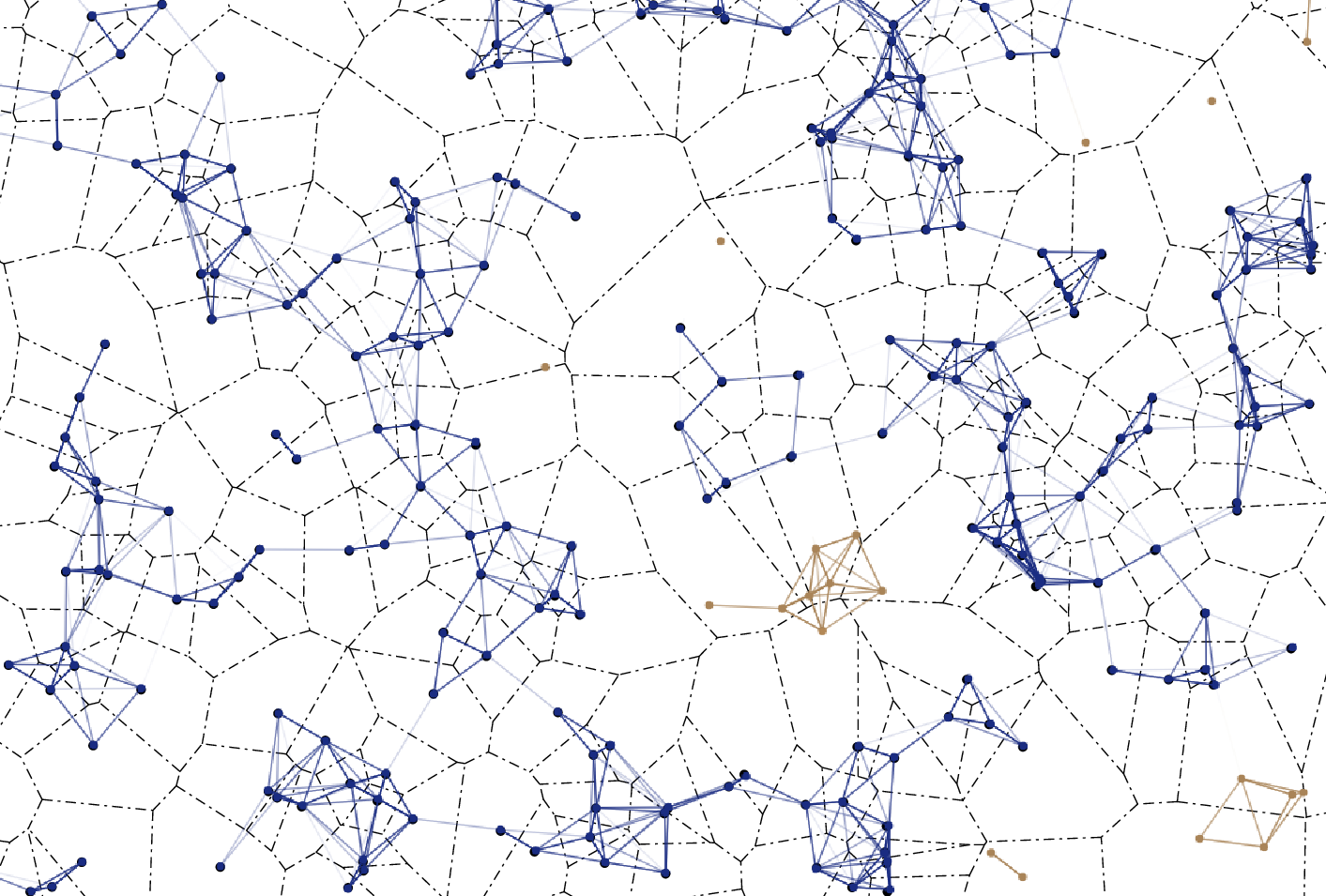}}
    \caption{A random geometric graph on $\R^2$ with the Voronoi partition generated by the infinite connected component $\Hh$ (in blue).}
    \label{fig:RGG_Voronoi}
\end{figure}

We now extend the domain of the $\Hh$-distance by defining $\dist(x,y):=\dist(q(x),q(y))$ for every $x,y \in \R^d$. The proposition below can be immediately adapted from the proof of~\citet[Theorem 2.2]{yao2011} by applying properties of Palm calculus and Lemma \ref{lm:ball.intersecting.infinite.comp}.
\begin{proposition}\hspace{-0.1cm}{\rm{(Adapted from~\citet[Theorem 2.2]{yao2011}).}} \label{prop:linearity.chem.dist}
    Let $d \geq 2$ and $\radius>\radius_c$. Then there exists $\rho_{\radius} > 0$ depending on $\radius$ such that, $\mu$-\textit{a.s.}, for all $x \in \R^d$,
    \[
        \lim_{\substack{\|y\|\uparrow + \infty}}\frac{\dist(x,y)}{\|y-x\|} = \rho_{\radius}.
    \]
\end{proposition}

The constant $\rho_{\radius}$ is called \emph{stretch factor} of $\Hh$. Observe that $\rho_{\radius} \geq 1/\radius$. Due to the subadditivity of the $\Hh$-distance, one can easily see that $\E_\mu[\dist(o, z)]$ with $\|z\|=1$ is an upper-bound for $\rho_{\radius}$.

We have the following result about the tail behaviour of $\dist(o, z)$.
\begin{lemma} \label{lm:chem.all}
    Let $d \geq 2$ and $\radius>\radius_c$. Then, there exist $c_1,c_2>0$ and $\beta' >1$ such that, for all $x \in \R^d$ and every $t> \beta' \|x\|$,
    \[
        \mu\big(\dist(o,x) \geq t\big)  \leq c_1\exp(-c_2 t).
    \]
\end{lemma}
\begin{proof}
    Let us write $\mu_{v,w}$ for the Palm measure $\mu \ast \delta_v \ast \delta_w$ for any given $v,w \in \R^d$. Set $\overline{\dist}$ to be the simple $\G_{\radius}$-distance. In what follows, $v \leftrightsquigarrow w$ stands for the existence of a path between $v$ and $w$ in $\G_{\radius}$. It is clear that $\dist(v,w)=\overline{\dist}(v,w)$ whenever $v,w\in V(\Hh)$. By \cite[Lemma 3.4]{yao2011}, there exist $\overline{c}_1,\overline{c}_2>0$ and $\beta'>1$ such that
    \begin{equation} \label{eq:Palm.bound}
        \mu_{v,w} \big(v \leftrightsquigarrow w , \text{~and~} \overline{\dist}(v,w) \geq t\big) \leq \overline{c}_1 \exp(-\overline{c}_2t)
    \end{equation}
    for all $t \geq \beta' \|v-w\|/2$. Consider now $\mathcal{B}_r(z):= B_r(z) \cap \mathcal{P}$. We apply Lemma \ref{lm:ball.intersecting.infinite.comp},  \eqref{eq:Palm.bound}, and Campbell's theorem to obtain that there exist $C,C'>0$ such that
    \begin{align} 
        \mu\big(\dist(o,x) \geq t\big) &\begin{multlined}[t]\leq \mu\big(\|q(o)\| \geq t/(2\beta')\big) + \mu\big(\|q(x)-x\| \geq t/(2\beta')\big) \\+ \mu\left(\bigcup_{\substack{v \in \mathcal{B}_{t/(2\beta')}(o), \, w \in \mathcal{B}_{t/(2\beta')}(x)}}\{v \leftrightsquigarrow w \text{~and~} \overline{\dist}(v,w) \geq t\}\right)\end{multlined} \nonumber\\
        &\leq 2 C \exp\big({-C't/(2\beta')}\big) + \overline{c}_1\frac{ \upupsilon_d^2}{2^{2d}}t^{2d} \exp({-\overline{c}_2t}) \label{eq:chem.prob.upper.bound}
    \end{align}
    for all $t \geq \beta' \|x\|$.
    Hence, we can conclude the proof of \eqref{eq:chem.prob.upper.bound} by suitably choosing $c_1,c_2>0$.
\end{proof}

Next, let us define for every $x \in \R^d$ the quantity
\[
    W_n^x := \{ \text{self-avoiding paths of length } n \text{ in } \G_{\radius} \text{ starting at }x\}
\]
and note that using \citet[Theorem 4.6.11]{jahnel2020}), we have that
\[
    \E_\mu[|W_n^x|]=(\upupsilon_dR^d)^n.
\]

Consider now $\big|W_n^{q(o)}\big|$, the number of self-avoiding paths in $\Hh$ starting at $q(o)$. We apply the previous result to prove the following lemma.

\begin{lemma} \label{lm:Wn.upperbound}
    Let $d \geq 2$, $\radius > \radius_c$, and $\kappa>1$. Then, one has $\mu$-a.s. that for sufficiently large $n \in \N$,
    \[
        \big|W_n^{q(o)}\big| < \left(\kappa ~ \upupsilon_d \radius^{d}\right)^n.
    \]
\end{lemma}
\begin{proof}
    Let $n\in\N$ and define the events
    \[
        \mathsf{A}_n := \left\{|W_n^{q(o)}| \geq (\kappa\upupsilon_d\radius^d)^n\right\}.
    \]
    Recall the notation $\mathcal{B}_n(o):= B_n(o) \cap \mathcal{P}$. Then
    \[
        \mathsf{A}_n \subseteq \Big(\bigcup_{ v \in \mathcal{B}_n(o)} \big\{|W_n^{v}| \geq (\kappa\upupsilon_d\radius^d)^n, \|q(o)\|< n \big\}\Big)\cup \{\|q(o)\| \geq n \}.
    \]
    Hence, by Markov's inequality, one has for any given $x \in \R^d$ that
    \[
        \mu\big(|W_n^x| \geq (\kappa\upupsilon_d\radius^d)^n\big) \leq \E_\mu \big[|W_n^x|\big]/(\kappa\upupsilon_d\radius^d)^n = 1/\kappa^n.
    \]
    Thus, by Lemma \ref{lm:ball.intersecting.infinite.comp} and Campbell's theorem, there exist $C,C'>0$ such that
    \[
        \sum_{n=1}^{+\infty}\mu(\mathsf{A}_n) \leq \upupsilon_d\sum_{n=1}^{+\infty}\frac{n^d}{\kappa^n} +  C\sum_{n=1}^{+\infty}{e^{-C'n^{d-1}}} < +\infty,
    \]
    and hence, an application of the Borel--Cantelli lemma completes the proof.
\end{proof}

\section{First-Passage Percolation} \label{sec:FPP}

We proceed to formally define our process. Let $\{\tau_e\}_{e \in E}$ be a family of independent and identically distributed random variables taking values in the time set $[0,+\infty)$. We say that $\tau_e$ is the passage time of the edge $e \in E(\Hh)$. 

Set $(\Om, \A, \p)$ to be the joint probability space associated with the construction of the random geometric graph $\G_{\radius}$ and the independent assignment of the random passage times $\{\tau_e\}_{e\in E}$. The joint probability space can be constructed as a product space and we will write $\E=\E_{\p}$ for short.

Given any path $\gamma=(x,x_1,\dots,x_m,y)\in \mathscr{P}(x,y)$ for $x,y \in V(\Hh)$, we write $e \in \gamma$ for an edge $e \in E(\Hh)$ between a pair of consecutive vertices of $\gamma$. We denote the passage time of the path $\gamma$ by
\[
    T(\gamma) := \sum_{e \in \gamma}\tau_e.
\]
The \emph{passage time} between $x,y \in \R^d$ is then defined by the random variable
\[
    T(x,y) := \inf\left\{ T(\gamma) : \gamma\in \mathscr{P}\big(q(x),q(y)\big) \right\}.
\]
In fact, we can see later that $T(x,y)$ is a random pseudo-metric when associated with a group action. To avoid cumbersome notation, we set $T(x):= T(o,x)$ for all $x \in \R^d$.

Using these definitions, we have 
\[
    H_t := \{x \in \R^d: T(x) \leq t\}
\]
for the set of Voronoi cells induced by $\Hh$ reached up to time $t$ with the FPP starting in $q(o)$.

Next, consider $\tht: \R^d \curvearrowright \Om$ to be an extension of the group action introduced in Section~\ref{sec:RGG} such that $\tht_z$ will induce $\tau_{\{x,y\}} \mapsto \tau_{\{x-z,y-z\}}$ independently in the product space. It is easily seen that $\tht$ inherits the ergodic property of the previously defined group action and that Remark \ref{rmk:isometry} still holds for actions on $\Om$ associated with isometries by extending them in the same fashion. Then, we observe that $T(x,x+y)= T(y)\circ\tht_x$ and that the subadditivity 
\begin{equation} \label{eq.subadditivity}
        T(x+y) \leq T(x) + T(y)\circ\tht_x
    \end{equation}
for all $x,y \in \R^d$ is straightforward.


We have the following lemma. 

\begin{lemma} \label{lm:FPP.mean.lower.bound}
    Let $d \geq 2$, $\radius>\radius_c$, and $\p(\tau=0) < 1/(\upupsilon_d \radius^d)$. Then, there exists $a>0$ depending on $\radius$ such that, for all $x \in \R^d$,
    \[
        a \|x\| \leq \E[T(x)].
    \]
\end{lemma}
\begin{proof}
    It suffices to prove this statement for large $\|x\|$ as due to subadditivity~\eqref{eq.subadditivity} and stationarity of $T(nx,(n+1)x)$ for all $n \in \N$ we have 
    \[
        {\E[T(mx)]}/{\|mx\|} \le \sum_{i=1}^m  \E[T((i-1)x,ix)]/\|mx\| = {\E[T(x)]}/{\|x\|}.
    \]
   Define the event
    \begin{align*}
        A^1_x:=  \big\{\max\{\|q(o)\|,\|q(x)-x\|\} \le {\|x\|}\big/4\big\}.
    \end{align*}
    In order to simplify the notation, let us abbreviate $m_x=\lceil\|x\|/(2\radius)\rceil$.
    Note that on~$A^1_x$, every path~$\gamma \in \mathscr{P}\big(q(o),q(x)\big)$ has~$|\gamma| \ge m_x$ and therefore includes a sub-path of length  at least $m_x$.
    Consequently, for any $t>0$,
    \begin{align*}
         \p\left(\left\{T(x) \le t \right\}\cap A^1_x \right) &= \p\Bigg(\Bigg\{ \inf_{\gamma \in \mathscr{P}\big(q(o),q(x)\big)} T(\gamma)\le t \Bigg\}\cap A^1_x \Bigg) \\
          &\le \p\left(\inf_{\gamma \in W_{m_x}^{q(o)}} T(\gamma) \le t\right),
    \end{align*}

    In order to proceed, first observe that, using Chernoff's bound for the binomial distribution with $X \sim \mathrm{Binomial}(n,p)$, we have
   \begin{align*}
        \p(X\le cn) &\le \exp\left(-n\left(c \log \frac{c}{p} + (1-c)\log\frac{1-c}{1-p}\right) \right) \\
        &=\left(p^{-c} (1-p)^{-(1-c)} \cdot c^{c} (1-c)^{(1-c)}    \right)^{-n},
   \end{align*}
   where for $c \rightarrow 0$ the base converges to $(1-p)^{-1}$. Also, because of the right continuity of the cumulative distribution function associated to $\tau$, there exist $\kappa>1$ and $\delta>0$ such that
    \begin{equation*}
        \p(\tau\leq \delta)<1/(\kappa \upupsilon_d \radius^d).
    \end{equation*}
Further, consider a random variable $X' \sim \mathrm{Binomial}(n,\p(\tau>\delta))$ with respect to $\p$.
    Note that $\tau \ge \delta \mathds{1}\{\tau > \delta\}$ and therefore on any (self-avoiding path) of length $|\gamma|=n$ the sum of $n$ i.i.d.~copies of $\tau$, stochastically dominates $\delta X'$.
    
    Therefore, there exists $c>0$ and $\kappa'>1$ such that for all $n\in \N$ and $|\gamma| = n$
    \[
        \p(T(\gamma) \le cn) \le   \p(X' > cn / \delta ) \leq  \big(\kappa' \upupsilon_d \radius^d\big)^{-n}.
    \]
     Fix $1<\kappa''<\kappa'$ and define
     \[
        A^2_x := \{|W_{m_x}^{q(o)}| \le (\kappa''\upupsilon_d \radius^d)^{m_x}\} \mathrm{\ and \ } A_{x} := A^1_x \cap A^2_x. 
     \]
    Then, we have
    \begin{align*}
        \p\left(T(x) \le  cm_x\right)&\le \p\big((A_x)^c\big) + \p\left({A_x} \cap \left\{{{\inf_{\gamma \in W_{m_x}^{q(o)}}}} T(\gamma) \le cm_x \right\}  \right)\\
        &\le  \p\big((A_x)^c\big) + \E\Big[\mathds{1}_{A_x} \cdot \sum_{\gamma \in W_{m_x}^{q(o)}}
         \hspace{-2pt}\mathds{1}\{T(\gamma) \le cm_x\}\Big]\\
        &\le \p\big((A_x)^c\big) +  (\kappa''\upupsilon_d \radius^d)^{m_x} \cdot (\kappa' \upupsilon_d \radius^d)^{-m_x},
    \end{align*}
    where $\p\big((A_x)^c\big)$ can be made arbitrarily small by choosing $\|x\|$ large enough via Lemma \ref{lm:ball.intersecting.infinite.comp} and Lemma \ref{lm:Wn.upperbound}. As $\kappa'' < \kappa'$, the exponent in the second summand is negative and dominates the polynomial term for large $x$. Therefore, there is a $k$ such that for all $x \in \R^d$ with $\|x\|\geq k$,   
    \[
        \p\big(T(x) \leq cm_x\big) \leq 1/2.
    \]
    Setting $a = c/4\radius$ we arrive at the statement $a\|x\| \leq \E[T(x)]$ when $\|x\| >k$.
\end{proof}

\begin{remark} \label{rmk:FPP.mean.upper.bound}
    Observe that {$\E[T(x)] \leq \E[\dist(o,x)]\E[\tau]$ due to the subadditivity and Fubini's Theorem. Moreover, condition \eqref{A2} implies that $\E[\tau] < +\infty$. One can easily see from Proposition \ref{prop:linearity.chem.dist} and $L^1$ convergence given by Kingman's subadditive ergodic theorem \cite{kingman1968} applied to the $\Hh$-distance, that, for all $x \in \R^d$,
    \[
        b:= \rho_{\radius}\E[\tau] \geq \limsup_{n \uparrow \infty} \E[T(nx)]/\|nx\|.
    \]}
\end{remark}

{
Denote by $\p_\xi$ the quenched probability of the propagation model given a realization $\xi \in \Upxi'$. The lemma below ensures the at-least linear growth of the passage times.
 
\begin{lemma} \label{lm:all.growth}
    Let $d \geq 2$, $\radius>\radius_c$, and assume that \eqref{A2} holds. Then, there exist deterministic $\beta>0$ and $\upkappa>1$ such that, for every $x,y \in \R^d$, and for each $\xi \in \Upxi'$,
    \[
        \p_{\xi}\big(T(x,y) \geq t\big) \leq  t^{-(d+\upkappa)}
    \]
    for all $t \geq \beta \dist(x,y)$.
\end{lemma}
\begin{proof} 
    Let $\gamma \in \mathscr{P}(x,y)(\xi)$ be a geodesic given by the $\Hh$-distance. Then, one has by Markov's inequality that, for $t > \E[\tau]\dist(x,y)$ and $\eta$ from \eqref{A2},
    \begin{equation} \label{eq:upper.geod}
        \p_\xi \big( T(x,y) \geq t \big) \leq \p_\xi \big( T(\gamma) \geq t \big) \leq \frac{\E_{\xi}\left[\left(\sum_{e \in \gamma}\left( \tau_e - \E[\tau]\right)\right)^\eta\right]}{\left(t- \E[\tau]\dist(x,y)\right)^\eta}.
    \end{equation}
    
    Rosenthal's inequality~\cite{rosenthal1970} states that, if~$Y_1,\ldots,Y_n$ are independent random variables with mean zero and finite moment of order~$p$, with~$p > 2$, then
\[
\E\left[\left(\sum_{i=1}^n Y_i\right)^p\right]\le C_p\cdot \max\left\{ \sum_{i=1}^n \E[|Y_i|^p];\;\left(\sum_{i=1}^n \E[(Y_i)^2]\right)^{p/2} \right\},
\]
where~$C_p > 0$ is a constant that depends only on~$p$.
Since \eqref{A2} holds and $(\tau_e-\E[\tau_e])$ are identically distributed for all $e \in \gamma$, it yields
\begin{equation}\label{eq:iid}
\E_{\xi}\Big[\Big(\sum_{e \in \gamma} (\tau_e-\E[\tau])\Big)^\eta\Big]\le C\dist(x,y)^{\eta/2},
\end{equation}
where~$C$ is now a constant that depends both on~$\eta$ and on the distribution of the random variables.

In case~$t \ge 2\E[\tau] \dist(x,y)$ we have~$t-\E[\tau] \dist(x,y) \ge t/2$. Using this and~\eqref{eq:iid}, the right-hand side of \eqref{eq:upper.geod} is smaller than
\[
2^\eta C \dist(x,y)^{\eta/2}t^{-\eta}.
\]
If we also have that $t \ge 4 C^{2/\eta} \dist(x,y)$, this is smaller than~$t^{-\eta/2}$. We have thus proved that
\[
\p_{\xi}\big(T(x,y) \geq t\big) \leq t^{-(d+\upkappa)} \quad \text{for all }t \ge \beta \dist(x,y),
\]
with $\beta := \max\left\{2\E[\tau], 4 C^{2/\eta} \right\}$ and $\upkappa := \eta/2-d>1$.
\end{proof}
}

Before proving our first main theorem, we state and prove the following result. It is an annealed version of the at least linear growth from lemma above in all directions.

\begin{lemma} \label{lm:FPP.seq.conv}
    Let $d \geq 2$, $\radius>\radius_c$. Consider the i.i.d.~FPP on the RGG satisfying \eqref{A2}. Then, there exist constants $\delta, C > 0$, and $\upkappa>1$ such that for all $t>0$ and all $x \in \R^d$, one has
   \[
        \p\left(\sup_{y \in B_{\delta t}(x)} T(x,y) \ge t\right) \le C t^{-\upkappa}.
    \]
\end{lemma}
\begin{proof}
    Due to the translation invariance it suffices to prove the lemma for $x = 0$. Let $\delta=(\beta'\beta)^{-1}$ with $\beta'$ and $\beta$ from Lemmas \ref{lm:chem.all} and \ref{lm:all.growth}. Set $c_d>0$ to be such that $B_{2\delta}(o) \subseteq [-c_d/2,c_d/2]^d$ and write $C_d:= 2 \theta_{\radius} c_d$  with $\theta_{\radius}>0$ from Proposition \ref{prop:Hn.growth}. Let us now define the following events:
    \begin{align*}
        G_1 &:= \big\{q(o) \in B_{\delta t}(o)\big\} \cap \big\{|B_{\delta 2t}(o) \cap V(\Hh)| \le C_d \cdot t^d \big\}\\
        G_2 &:= \left\{\sup_{\|y\| \le \delta t}  \dist(o,y) \le t/ \beta\right\}\cap G_1\\
        G_3 &:= \left\{\sup_{\|y\| \le \delta t} T(o,y) \ge t\right\} \cap G_1 \cap G_2
    \end{align*}
        Now,
    \begin{equation} \label{eq:prob.sup.bound}
        \p\left(\sup_{y \in B_{\delta t}(o)} T(o,y) \ge t\right)\le \p(G_3) + \p(G_1^{\mathrm{c}}) +  \p(G_2^{\mathrm{c}}),
    \end{equation}
    where the last two summands decrease exponentially in $t$ due to Proposition \ref{prop:Hn.growth}, and Lemmas \ref{lm:ball.intersecting.infinite.comp} and \ref{lm:chem.all}. By Lemma \ref{lm:all.growth}, there exists $\upkappa>1$ such that
    \begin{equation} \label{eq:prop.G3.bound}
        \p(G_3) \le \E \left[ \mathds{1}\{\xi\in G_1 \cap G_2\} \p_{\xi} \left( \sup_{\|y\| < \delta t} T(y) \ge t \right) \right] \le  C_d t^d/ t^{d+\upkappa}.
    \end{equation}
    Combining \eqref{eq:prop.G3.bound} with \eqref{eq:prob.sup.bound}, the desired bound is obtained by choosing a suitable $C>0$.
\end{proof}

After this preparatory work, we now proceed to prove Theorem~\ref{thm:shape.FPP}. The methods are closely related to standard techniques for shape theorems which can be found in \cite{kesten1986}, for instance.
\begin{proof}[Proof of Theorem \ref{thm:shape.FPP}]
    We begin by verifying properties of $T(nx)$. Note that for every $x \in \R^d$ one has that $\E[T(x)] < +\infty$ by Lemmas \ref{lm:chem.all} and \ref{lm:all.growth}. Recall that the process is mixing on $(\Om, \A,\p,\tht)$ by Lemma \ref{lm:PPP.mixing}. Then, by the subadditivity \eqref{eq.subadditivity}, we apply Kingman's subadditive ergodic theorem to obtain that, $\p$-\textit{a.s.}, for all $x \in \R^d$,
    \begin{equation} \label{eq.FPP.conv.T(nx)}
        \lim_{n\uparrow+\infty}\frac{T(nx)}{n} = \phi (x),
    \end{equation}
    where $\phi:\R^d\to [0,+\infty)$ is a homogeneous and subadditive function given by
    \[
        \phi(x) = \inf_{n\geq 1} \frac{\E[T(nx)]}{n} = \lim_{n\uparrow+\infty} \frac{\E[T(nx)]}{n}.
    \]
    
    Since the process is rotation invariant, there exists a constant $\upvarphi$ (the time constant) such that $\phi(x)=\upvarphi^{-1}\|x\|$ for all $x \in \R^d$. In fact, one has from Lemma \ref{lm:FPP.mean.lower.bound} and Remark \ref{rmk:FPP.mean.upper.bound} that
    \[
        0<a \leq \upvarphi^{-1} \leq b= \rho_{\radius}\E[\tau] < \infty.
    \]

   Let us now prove the $\p$-\textit{a.s.} asymptotic equivalence
   \begin{equation} \label{eq:asymptotic.equivalence}
        \lim_{\|y\|\uparrow+\infty} \frac{T(y)}{\|y\|} = \frac{1}{\upvarphi}.
    \end{equation}

    For the approach from below, we prove the equivalent statement that for every $\epsilon \in (0,1)$,
    \[
        \limsup_{s \uparrow \infty}\quad \sup_{\| y \| \le (1-\epsilon)s}\frac{T(y)}{s} = \limsup_{m \in \N ,m \uparrow \infty}\quad \sup_{\| y \| \le (1-\epsilon)m}\frac{T(y)}{m} \le \frac{1}{\upvarphi} \quad \p-a.s.,
    \]
    where the first equation holds as $\lfloor s \rfloor/s$ converges to $1$.
    Fix $\epsilon \in (0,1)$  and let $\delta$ be given by Lemma \ref{lm:FPP.seq.conv}. Due to compactness, there exists a finite cover of open balls with centers $(y_i)_{i \in \{1,\dots,n\}} \subseteq \R^d$ with $\|y_i\| \le 1-\epsilon$ such that
    \[
        \overline{B_{1 - \epsilon}(o)} \subseteq \bigcup_{{i \in \{1,\dots,n}\}} B_{\delta\epsilon/(2\upvarphi)}(y_i).
    \]
    Furthermore $B_{m(1-\epsilon)}(o) \subseteq \bigcup_{{i \in \{1,\dots,n}\}} B_ {m\delta\epsilon/(2\upvarphi)}(my_i)$ for every $m \in \N$.
    Applying Lemma $\ref{lm:FPP.seq.conv}$ we obtain
    \[
       \sum_{m \in \N} \p\left( \sup_{\|y- my_i\| \le m \delta\epsilon/(2\upvarphi)} T(m y_i, y) > m \epsilon/(2\upvarphi) \right) < \infty.
    \]
    Therefore, by the Borel--Cantelli lemma,
    \[
      \limsup_{m \in \N, m \uparrow \infty} \quad \sup_{\|my_i-y\| \le m \delta\epsilon/(2\upvarphi)} \frac{T(my_i,y)}{m} <  \frac{\epsilon}{2\upvarphi} \quad \p-a.s.
    \]
    Applying \eqref{eq.FPP.conv.T(nx)} and subadditivity, we obtain
    \begin{align*}
       \limsup_{m \uparrow \infty; \| y \| \le (1-\epsilon)m}\frac{T(y)}{m} &\le 
       \limsup_{m \uparrow \infty} \begin{multlined}[t]\bigg(\max_{i \in \{1,\dots n\}} \frac{T(o,my_i)}{m} \\ \hspace{50pt}+ \sup_{\|my_i-y\| \le m \delta\epsilon/(2\upvarphi)} \frac{T(my_i,y)}{m}\bigg)\end{multlined}\\
       &\le \max_{i \in \{1,\dots, n\}}\|y_i\|/\upvarphi + \epsilon/(2\upvarphi) < 1/\upvarphi \quad \p-a.s.,
    \end{align*}
    where we used that $\|y_i\| < 1-\epsilon$.

    For the approach from above, define  $A_t := B_{t(1+2\epsilon)}(o) \setminus B_{t(1+\epsilon)}(o)$ 
   and observe that it suffices to prove
    \[
        \liminf_{m \in \N, m \uparrow \infty}\quad \inf_{y \in {A}_t}\frac{T(y)}{m} \ge \frac{1}{\upvarphi} \quad \p-a.s.
    \]
    for arbitrary but fixed $\epsilon>0$, as for (for $t > \epsilon$) any $x$ with $\|x\|>t(1+2\epsilon)$ there exists an $\tilde x \in A_t$ with $T(\tilde x) \le T(x)$.
    
    Similar as in the approach from below, fix $\epsilon > 0$ and $\delta > 0$ small enough that Lemma \ref{lm:FPP.seq.conv} holds. There exists a set of centers $(y_i)_{i \in \{1,\dots,n\}} \subseteq \R^d$ with $\|y_i\| \ge 1+\epsilon$ such that
        \[
        A_t \subseteq \bigcup_{{i \in \{1,\dots,n}\}} B_{\delta\epsilon/(2\upvarphi)}(y_i)
    \]
and hence
    \begin{align*}
       \liminf_{m\in \N,m \uparrow \infty}\quad\inf_{y \in {A}_m}\frac{T(y)}{m} &\ge 
        \liminf_{m\in \N,m \uparrow \infty} \begin{multlined}[t]\bigg(\min_{i \in \{1,\dots n\}} \frac{T(o,my_i)}{m} \\ \hspace{40pt} - \sup_{\|my_i-y\| \le m \delta\epsilon/(2\upvarphi)} \frac{T(my_i,y)}{m}\bigg)\end{multlined}\\
       &\ge\min_{i \in \{1,\dots n\}}\|y_i\|/\upvarphi - \epsilon/(2\upvarphi) > 1/\upvarphi,
    \end{align*}
    which concludes the proof of the asymptotic equivalence \eqref{eq:asymptotic.equivalence}. The proof of the theorem is now complete by standard arguments of the $\p$-\textit{a.s.} uniform convergence given by \eqref{eq:asymptotic.equivalence}.
\end{proof}

\section{Proof of Theorem~\ref{thm:rescaling.FPP}}\label{sec:Scale}

Throughout this section, we fix~$r,\lambda,\lambda_{\rm I}$.

We start this section by giving some details on the topologies involved in the statement of Theorem~\ref{thm:rescaling.FPP}, and introducing some notation.

Let~$\mathcal{M}$ denote the space of measures of the form~$\mu = \sum_{i=1}^k \delta_{\{z_i\}}$, where~$k \in \mathbb{N}_0$ and~$z_1,\ldots, z_k \in \mathbb{R}^d$ are distinct. We endow this space with the weak topology, under which a sequence~$\mu_n$ converges to~$\mu$ if and only if~$\int f\;\mathrm{d}\mu_n \xrightarrow{n \to \infty} \int f\;\mathrm{d}\mu$ for all~$f:\mathbb{R}^d \to \mathbb{R}$ that is continuous and bounded. This topology is metrizable; see \cite{prokhorov1956convergence}. 
A sequence~$\mu_n$ converges to~$\mu = \sum_{i=1}^k \delta_{\{z_i\}}$  in this topology if and only if the following two conditions are satisfied: first, for~$n$ large enough, the total masses agree, that is,~$\mu_n(\mathbb{R}^d)=\mu(\mathbb{R}^d) = k$, and second, we can take an enumeration (for~$n$ large enough)~$\mu_n = \sum_{i=1}^k \delta_{\{z_{n,i}\}}$ so that~$z_{n,i} \xrightarrow{n \to \infty} z_i$ for each~$i$.

We let~$D_0$ be the space of all functions~$\gamma:[0,\infty) \to \mathcal{M}$ of the form
\begin{equation}\label{eq_form}
\gamma(t) =  \sum_{k=0}^ \infty \mathds{1}\{s_k\le t\}\delta_{\{z_k\}} ,\quad t \ge 0,
\end{equation}
where~$z_0,z_1,\ldots \in \mathbb{R}^d$ are distinct,~$0=s_0 < s_1 < s_2 < \cdots$, and~$s_n \xrightarrow{n \to \infty} \infty$. For~$\gamma$ of this form, we let
\begin{equation}\label{eq_projs}
\phi_k(\gamma) = z_k,\; k \ge 0 \qquad \text{and}\qquad \psi_k(\gamma) =  s_k - s_{k-1}, \; k \ge 1.
\end{equation}
We endow~$D_0$ with the Skorokhod topology; see \cite{ethier2009markov}, Chapter 3. Note that this gives rise to a topological subspace of the more usual space~$D$ of c\`adl\`ag functions; since the processes we are considering have constant-by-parts trajectories, it is more natural for us to work on~$D_0$ than in~$D$.
 By the definition of the Skorokhod topology, it is easy to see that
\begin{equation}\begin{split}\label{eq_cond_conv}
&\gamma_n \xrightarrow{n \to \infty} \gamma \quad \text{if and only if} \\[.2cm] &\hspace{2cm}\phi_k(\gamma_n)\xrightarrow{n \to \infty} \phi_k \;\forall k \ge 0,\; \psi_k(\gamma_n) \xrightarrow{n \to \infty} \psi_k(\gamma) \; \forall k \ge 1. 
\end{split}\end{equation}
Now let
\[\Lambda(\gamma):= (\phi_0(\gamma),\phi_1(\gamma),\psi_1(\gamma),\phi_2(\gamma),\psi_2(\gamma),\ldots),\quad \gamma \in D_0.\]
Note that~$\Lambda$ is a one-to-one mapping from~$D_0$ to~$\Lambda(D_0) \subset \mathbb{R}^d \times (\mathbb{R}^d \times (0,\infty))^\mathbb{N}$. We endow~$\Lambda(D_0)$ with the product topology (under which convergence means convergence in each coordinate, with respect to the Euclidean topology). With this choice, by~\eqref{eq_cond_conv},~$\Lambda$ is a homeomorphism between~$D_0$ and~$\Lambda(D_0)$. 

We now return to the processes~$(\Hh_{t}^{\alpha \lambda, \lambda_{\rm I}/\alpha})_{t \ge 0}$ and $(\mathcal{T}_t^{\lambda,\lambda_\mathsf{I}})_{t \ge 0}$ and note that, for any $t\ge 0$, $\Hh_{t}^{\alpha \lambda, \lambda_{\rm I}/\alpha}$ can be written as $\sum_{z_i\in \Hh}\mathds{1}\{s_i\le t\}\delta_{\{z_i\}}$ where $s_i$ is the (random) time at which the point $z_i$ is first reached in the FPP model. In particular, by Theorem~\ref{thm:shape.FPP}, $\Hh_{t}^{\alpha \lambda, \lambda_{\rm I}/\alpha}\in\mathcal M$ and the same holds for $\mathcal{T}_t^{\lambda,\lambda_\mathsf{I}}$. 
 Now, write
 \begin{equation*}
 Z_{\alpha,k}:= \phi_k((\mathcal{H}^{\alpha \lambda, \lambda_\mathsf{I}/\alpha}_t)_{t \ge 0}),\quad Z_k := \phi_k((\mathcal{T}_t^{\lambda,\lambda_\mathsf{I}})_{t \ge 0}),\quad k \ge 0
 \end{equation*}
 and note that~$Z_{\alpha,0} = q(o)$ and~$Z_0 = o$. Also write~$T_{\alpha,0} = T_0 := 0$ and
 \begin{equation*}
      T_{\alpha,k}:= \psi_k((\mathcal{H}^{\alpha \lambda, \lambda_\mathsf{I}/\alpha}_t)_{t \ge 0}),\quad T_k := \psi_k((\mathcal{T}_t^{\lambda,\lambda_\mathsf{I}})_{t \ge 0}),\quad k \ge 1.
 \end{equation*}
 With this notation, we can state:
\begin{proposition} \label{prop_prox_kernel}
	Let~$k \in \mathbb{N}$,~$f_0,\ldots, f_k : \mathbb{R}^d \to \mathbb{R}$ and~$g_1,\ldots, g_k:(0,\infty) \to \mathbb{R}$; assume all these functions are continuous with compact support.
 We then have
	\begin{equation}\label{eq_conv_of_exp}
		\E\left[f_0(Z_{\alpha,0}) \prod_{i=1}^k f_i(Z_{\alpha,i})g_i(T_{\alpha,i})\right] \xrightarrow{\alpha \to \infty}f_0(o)\cdot 	\E\left[ \prod_{i=1}^k f_i(Z_{i})g_i(T_{i})\right].
	\end{equation}
\end{proposition}

We will prove this proposition later; for now let us show how it implies Theorem~\ref{thm:rescaling.FPP}.
\begin{proof}[Proof of Theorem~\ref{thm:rescaling.FPP}]
Proposition~\ref{prop_prox_kernel} and standard approximation arguments imply that, for all~$k$,
\[(Z_{\alpha,0},Z_{\alpha,1},T_{\alpha,1},\ldots,Z_{\alpha,k},T_{\alpha,k}) \xrightarrow{\alpha \to \infty} (Z_0,Z_1,T_1,\ldots Z_k,T_k)\]
in distribution. This implies that
\begin{equation} \label{eq_conv_distr_v} (Z_{\alpha,0},Z_{\alpha,1},T_{\alpha,1},Z_{\alpha,2},T_{\alpha,2},\ldots) \xrightarrow{\alpha \to \infty} (Z_0,Z_1,T_1,Z_2,T_2,\ldots)\end{equation}
in distribution, because convergence in distribution in the infinite product topology is equivalent to convergence of all finite-dimensional distributions. Given a function~$h: D_0 \to \mathbb{R}$ that is continuous and bounded, we have that
\begin{align*}
&\mathbb{E}[h((\mathcal{H}_t^{\alpha \lambda,\lambda_\mathrm{I}/\alpha})_{t \ge 0})]= \mathbb{E}[h(\Lambda^{-1}(Z_{\alpha,0},Z_{\alpha,1},T_{\alpha,1},Z_{\alpha,2},T_{\alpha,2},\ldots))]\\[.2cm]
&\hspace{2.8cm} \xrightarrow{\alpha \to \infty} \mathbb{E}[h(\Lambda^{-1}(Z_0,Z_1,T_1,Z_2,T_2,\ldots))] = \mathbb{E}[h((\mathcal{T}_t^{\lambda,\lambda_\mathrm{I}})_{t \ge 0})],
\end{align*}
where the convergence follows from~\eqref{eq_conv_distr_v} and the fact that~$h \circ \Lambda^{-1}$ is continuous and bounded. This proves that~$(\mathcal{H}_t^{\alpha \lambda,\lambda_\mathrm{I}/\alpha})_{t \ge 0}$ converges to~$(\mathcal{T}_t^{\lambda,\lambda_\mathrm{I}})_{t \ge 0}$ in distribution.
\end{proof}

In order to prove Proposition~\ref{prop_prox_kernel}, it will be important to have a more explicit description of the expectations that appear in~\eqref{eq_conv_of_exp}. To this end, let us give some definitions. 
Recall that~$\mathcal{P}_{\alpha \lambda}$ denotes a PPP on~$\mathbb{R}^d$ with intensity~$\alpha  \lambda$; we assume that this is the point process that gives rise to the infinite cluster in which the growth process~$(\Hh_{t}^{\alpha \lambda, \lambda_{\rm I}/\alpha})_{t \ge 0}$ is defined.
Given a realization of~$\mathcal{P}_{\alpha \lambda}$ and a finite set~$S \subseteq \mathcal{P}_{\alpha\lambda}$, define
\[\mathcal{N}_\alpha(S):= \sum_{x \in S} \big|(\mathcal{P}_{\alpha\lambda} \cap B_r(x) )\backslash S\big|= \sum_{y \in \mathcal{P}_{\alpha \lambda}\backslash S}\big|\{x \in S:\;\|x-y\|\le r\}\big|.\]
We now introduce {probability kernels} for each value of~$\alpha$; these encode the jump rates of the dynamics of the growth process, conditioned on the realization of~$\mathcal{P}_{\alpha\lambda}$. We start with the temporal kernel
\[\mathcal{L}_\alpha(S,\mathrm{d}t):= \frac{\mathcal{N}_\alpha(S)\lambda_{\rm I}}{\alpha}\cdot \exp\left(-\frac{\mathcal{N}_\alpha(S)\lambda_{\rm I}}{\alpha}\cdot t\right)\mathrm{d}t,\]
where~$S$ is any finite subset of~$\mathcal{P}_{\alpha\lambda}$ and~$\mathcal{L}_\alpha(S,\cdot)$ gives a measure on the Borel sets of~$(0,\infty)$, described above in terms of its density with respect to Lebesgue measure.
Next, define the spatial kernel
\[\mathcal{K}_\alpha(S,A) := \frac{1}{\mathcal{N}_\alpha(S)}\cdot \sum_{y \in \mathcal{P}_{\alpha \lambda}\backslash S}\big|\{x \in S:\;\|x-y\|\le r\}\big|\cdot \delta_{\{y\}}(A).\]
where~$S$ is any finite subset of~$\mathcal{P}_{\alpha\lambda}$ and~$A$ is a Borel subset of~$\mathbb{R}^d$, so that~$\mathcal{K}_\alpha(S,\cdot)$ gives a probability measure on~$\mathbb{R}^d$. Finally, define the kernel
\[K_\alpha(S,\mathrm{d}(x,t)) := \mathcal{K}_\alpha(S,\mathrm{d}x) \otimes \mathcal{L}_\alpha(S,\mathrm{d}t),\]
that is,~$K_\alpha(S,\cdot)$ is the measure on~$\mathbb{R}^d \times (0,\infty)$ that satisfies, for any functions~$f:\mathbb{R}^d \to \mathbb{R}$ and~$g:(0,\infty) \to \mathbb{R}$ both continuous with compact support:
\[\int_{\mathbb{R}^d\times(0,\infty)} f(x)g(t)\;K_\alpha(S,\mathrm{d}(x,t)) = \int_{\mathbb{R}^d} f(x)\mathcal{K}_\alpha(S,\mathrm{d}x) \times \int_{(0,\infty)} g(t)\;\mathcal{L}_\alpha(S,\mathrm{d}t).\]
Now, using the strong Markov property, we can write
\begin{equation}\label{eq_long_kernel}\begin{split}
	&\E\left[f_0(Z_{\alpha,0})\prod_{i=1}^k f_i(Z_{\alpha,i})g_i(T_{\alpha,i})\right]  \\
	&=\E\left[f_0(q(o))  \int K_\alpha(\{q(o)\},\mathrm{d}(z_1,t_1)) f_1(z_1)g_1(t_1)\right.\\
	&\hspace{2.3cm}\int K_\alpha(\{q(o),z_1\},\mathrm{d}(z_2,t_2))f_2(z_2)g_2(t_2)\cdots \\
&\hspace{2.3cm}\left.\int K_\alpha(\{q(o),z_1,\ldots,z_{k-1}\},\mathrm{d}(z_k,t_k)) f_k(z_k)g_k(t_k)\right].\end{split}
\end{equation}

We now define the analogous kernels for the limiting growth process. The temporal kernel is given by
\[\mathcal{L}(S,\mathrm{d}t) := |S|\upupsilon_dr^d\lambda \lambda_{\rm I} \cdot\exp\big(- |S|\upupsilon_dr^d\lambda \lambda_{\rm I} \cdot t\big)\; \mathrm{d}t,\]
where~$S \subseteq \mathbb{R}^d$ is finite, and again we obtain a measure on~$(0,\infty)$. Next, the spatial kernel is given by
\[\mathcal{K}(S,\mathrm{d}y) := \left( \frac{1}{|S|\upupsilon_dr^d}\sum_{x \in S} \mathds{1}\{y \in B_r(x)\}\right)\mathrm{d}y,\]
for~$S \subseteq \mathbb{R}^d$ finite. So~$\mathcal{K}(S,\cdot)$ is the probability measure on~$\mathbb{R}^d$ obtained from first choosing~$x \in S$ uniformly at random, and then choosing a point~$y \in B_r(x)$ uniformly at random. We again let~$K(S,\mathrm{d}(x,t)):= \mathcal{K}(S,\mathrm{d}x)\otimes \mathcal{L}(S,\mathrm{d}t)$. Again by the Markov property, the equality in~\eqref{eq_long_kernel} holds for the limiting process with respect to this kernel (that is, the same equality with~$q(o)$ replaced by~$o$ and all~$\alpha$'s omitted).

The following will be the essential ingredient in the proof of Proposition~\ref{prop_prox_kernel}.

\begin{lemma}\label{lem_prox_kernel}
\label{lem_prox_kernel_L}
	Let~$f:\mathbb{R}^d \to \mathbb{R}$ and~$g:(0,\infty) \to \mathbb{R}$ be both continuous  with compact support,~$k\in \mathbb{N}$ and~$\varepsilon > 0$. There exists~$\alpha_0 > 0$ such that for any~$\alpha \ge \alpha_0$, we have that with probability larger than~$1-\varepsilon$,~$\mathcal{P}_{\alpha\lambda}$ satisfies the following. For any set~$S \subseteq B_{kr}(o) \cap \mathcal{P}_{\alpha\lambda}$ with~$|S|\le k$, we have
	\[\left|\int_{\mathbb{R}^d}f(x) g(t)\;K_\alpha(S,\mathrm{d}(x,t))-\int_{\mathbb{R}^d}f(x) g(t)\;K(S,\mathrm{d}(x,t))\right| < \varepsilon.\]
\end{lemma}
We postpone the proof of this lemma; it will immediately follow from Lemma~\ref{lem_prox_kernel_K} and Lemma~\ref{lem_prox_kernel_LL} below, which separately treat the spatial and temporal kernels. For now, let us show how Lemma~\ref{lem_prox_kernel}  implies Proposition~\ref{prop_prox_kernel}. 
\begin{proof}[Proof of Proposition~\ref{prop_prox_kernel}]
	We will abbreviate
	\[[K_\alpha fg](S):= \int_{\mathbb{R}^d \times (0,\infty)} f(x)g(t)K_\alpha(S,\mathrm{d}(x,t)),\]
	and also
	\[S_{\alpha,k}:= \{Z_{\alpha,0},Z_{\alpha,1},\ldots, Z_{\alpha,k}\},\quad k \in \mathbb{N}_0,\]
	and similarly when~$\alpha$ is absent.

	We proceed by induction on~$k\in \mathbb{N}_0$. We interpret the case~$k= 0$ to mean that
	\begin{equation*}
		\E[f_0(Z_{\alpha,0})] \xrightarrow{\alpha \to \infty} f_0(0),
	\end{equation*}
	which holds because~$Z_{\alpha,0}=q(o)$ converges in probability to~$o$ as~$\alpha \to \infty$, as is easily seen. Now assume that the statement has been proved for~$k-1 \ge 0$, and take functions $f_1,\ldots, f_{k},g_1,\ldots, g_{k}$ as in the statement. It will be convenient to add and subtract as follows (for~$k=1$, we interpret~`$\prod_{i=1}^0$' as being equal to one): 
	\begin{align*}
		&\E\left[f_0(Z_{\alpha,0}) \prod_{i=1}^{k-1} f_i(Z_{\alpha,i})g_i(T_{\alpha,i})\right]\\
		&= \E\left[ f_0(Z_{\alpha,0})\prod_{i=1}^{k-1} f_i(Z_{\alpha,i})g_i(T_{\alpha,i})\cdot [K_\alpha f_k g_k](S_{\alpha,k})\right]\\
		& =\E\left[ \prod_{i=1}^{k-1} f_i(Z_{\alpha,i})g_i(T_{\alpha,i})\cdot \big([K_\alpha f_{k}g_{k}](S_{\alpha,k})\pm [K f_{k}g_{k}](S_{\alpha,k})\big)\right].
	\end{align*}
	Noting that the function that maps~$(z_0,z_1,t_1,\ldots, z_{k-1},t_{k-1})$ into
	\[f_0(z_0)\prod_{i=1}^{k-1}f_i(z_i)g_i(t_i)\cdot [Kf_kg_k](\{z_0,z_1,\ldots,z_{k-1}\}) \]
	is continuous, the induction hypothesis and the definition of weak convergence give
	\begin{align*}
		&\E\left[f_0(Z_{\alpha,0}) \prod_{i=1}^{k-1} f_i(Z_{\alpha,i})g_i(T_{\alpha,i}) [K f_k g_k](S_{\alpha,k})\right] \\
		&\xrightarrow{\alpha \to \infty}\E\left[f_0(Z_0) \prod_{i=1}^{k-1} f_i(Z_{i})g_i(T_{i}) [K f_k g_k](S_{k})\right] = \E\left[f_0(Z_0) \prod_{i=1}^k f_i(Z_i)g_i(T_i)\right].
	\end{align*}
	Next, we bound
	\begin{align*}
		&\E\left[f_0(Z_{\alpha,0}) \prod_{i=1}^{k-1} f_i(Z_{\alpha,i})g_i(T_{\alpha,i})\cdot |[K_\alpha f_{k}g_{k}](S_{\alpha,k})- [K f_{k}g_{k}](S_{\alpha,k})|\right]\\
		&\le \left( \max_{i \le k-1} (\|f_i\|_\infty\vee \|g_i\|_\infty)\right)^{k}\cdot \E\left[ |[K_\alpha f_{k}g_{k}](S_{\alpha,k})- [K f_{k}g_{k}](S_{\alpha,k})|\right].
	\end{align*}
	By Lemma~\ref{lem_prox_kernel}, the right-hand side converges to zero as~$\alpha \to \infty$. This completes the proof.
\end{proof}

It remains to prove Lemma~\ref{lem_prox_kernel}. To do so, let us now introduce some more notation.
For each~$\delta > 0$, we define the collection of cubes
\[\mathcal{C}_\delta := \left\{\delta z + [-\tfrac{\delta}{2},\tfrac{\delta}{2})^d:\;z \in \mathbb{Z}^d\right\}.\]
 Additionally, given~$\alpha > 0$,~$\ell \in \mathbb{N}$ and~$\delta > 0$, we define the event (involving the set~$\mathcal{P}_{\alpha\lambda}$, but not the passage times):
\[\mathrm{REG}_\alpha(\ell,\delta):= \left\{\left| \frac{|\mathcal{P}_{\alpha\lambda} \cap Q|}{\alpha \lambda\delta^d} - 1\right| < \delta \text{ for all } Q \in \mathcal{C}_\delta \text{ with } Q \subseteq B_\ell(o)\right\}.\]
By the law of large numbers we have
\begin{equation}\label{eq_lim_of_reg}\lim_{\alpha \to \infty} \p\big(\mathrm{REG}_\alpha(\ell,\delta)\big) = 1.\end{equation}

\begin{lemma}\label{lem_helpful_kernel}
	For any~$\varepsilon > 0$ and~$k \in \mathbb{N}$ there exists~$\delta_1 = \delta_1(\varepsilon,k)$ such that the following holds for~$\delta \in (0,\delta_1]$. For any~$\ell \ge r$, if~$\alpha$ is large enough and the event~$\mathrm{REG}_\alpha(\ell+1,\delta)$ occurs, then for any~$S \subseteq\mathcal{P}_{\alpha\lambda}\cap B_{\ell-r}(o) $ with~$|S| \le k$ we have
	\begin{equation}\label{eq_big_abs}	
		\sum_{x \in S}\sum_{Q\in \mathcal{C}_\delta} \left| \frac{\big|(\mathcal{P}_{\alpha\lambda}\cap Q \cap B_r(x) )\backslash S\big|}{\alpha \lambda} - \int_Q \mathds{1}\{\|y-x\| \le r\}\;\mathrm{d}y\right| < \varepsilon.
	\end{equation}
\end{lemma}
\begin{proof}
It is not hard to see that we can choose~$\delta' > 0$ so that for any~$\delta < \delta'$ we have
	\begin{equation}\label{eq_boundary_cubes}
		\delta^d\sup_{x \in \mathbb{R}^d} \sum_{Q \in \mathcal{C}_\delta} \mathds{1}\{Q \cap \partial B_r(x) \neq \varnothing\}< \frac{\varepsilon}{3k}, 
	\end{equation}
	where~$\partial B_r(x):= \{y \in \mathbb{R}^d:\|x-y\|= r\}$. Next, we let~$\delta_1:= \min\left(\delta',\; \frac{\varepsilon}{2\upupsilon_dr^d}\right)$.

	Now, assume that~$\delta < \delta_1$ and that~$\mathrm{REG}_\alpha(\ell+1,\delta)$ occurs. Fix~$S \subseteq \mathcal{P}_{\alpha \lambda} \cap B_{\ell-r}(o)$ with~$|S| \le k$, and also fix~$x \in S$. For each~$Q \in \mathcal{C}_\delta$ define
	\[ \mathscr{E}_x(Q):= \left| \frac{|(\mathcal{P}_{\alpha\lambda}\cap Q \cap B_r(x) )\backslash S|}{\alpha \lambda} - \int_Q \mathds{1}\{\|y-x\| \le r\}\;\mathrm{d}y\right|. \]
	If~$Q \cap \partial B_r(x) \neq \varnothing$ we bound
	\begin{align*}
		\mathscr{E}_x(Q)\le \frac{|\mathcal{P}_{\alpha\lambda} \cap Q|}{\alpha \lambda} + \delta^d \le (1+\delta)\delta^d + \delta^d =(2+\delta)\delta^d,
	\end{align*}
	by the triangle inequality and the definition of~$\mathrm{REG}_\alpha(\ell,\delta)$. If~$Q \subseteq B_r(x)$ with~$Q \cap \partial B_r(x) = \varnothing$ we have
	\[\frac{|\mathcal{P}_{\alpha\lambda} \cap Q \cap B_r(x)|}{\alpha \lambda} = \frac{|\mathcal{P}_{\alpha\lambda} \cap Q|}{\alpha \lambda}\in ((1-\delta)\delta^d,(1+\delta)\delta^d), \]
	so we bound
	\[\mathscr{E}_x(Q) \le \frac{k}{\alpha\lambda}+\delta\cdot \delta^d \le 2\delta^{d+1}\]
	(the factor~$\tfrac{k}{\alpha \lambda}$ is there to account for the possibility that~$Q$ contains some points of~$S$; the second inequality holds if~$\alpha$ is large enough that~$k/(\alpha \lambda) < \delta^{d+1}$).
	Now, also using~\eqref{eq_boundary_cubes}, we have that the left-hand side of~\eqref{eq_big_abs} is at most
	\begin{align*} \sum_{x \in S} \sum_{Q \in \mathcal{C}_\delta} \mathscr{E}_x(Q)&\le 
		\sum_{x \in S} \left((2+\delta)\delta^d  \sum_{Q \in \mathcal{C}_\delta} \mathds{1}\{Q \cap \partial B_r(x) \neq \varnothing\}+ 2\delta^{d+1} \frac{\upupsilon_d r^d}{\delta^d}\right)\\
		&\le k\cdot \left( \frac{(2+\delta)\varepsilon}{3k} + 2\delta^{d+1} \frac{\upupsilon_d r^d}{\delta^d}\right).
	\end{align*}
	If~$\delta$ is small enough, the right-hand  side is smaller than~$\varepsilon$, completing the  proof.
\end{proof}

\begin{lemma}\label{lem_prox_kernel_K}
	Let~$f:\mathbb{R}^d \to \mathbb{R}$ be continuous with compact support,~$k\in \mathbb{N}$ and~$\varepsilon > 0$. There exists~$\alpha_0 > 0$ such that for any~$\alpha \ge \alpha_0$, we have that with probability larger than~$1-\varepsilon$,~$\mathcal{P}_{\alpha\lambda}$ satisfies the following. For any set~$S \subseteq B_{kr}(o) \cap \mathcal{P}_{\alpha\lambda}$ with~$|S|\le k$, we have
	\[\left|\int_{\mathbb{R}^d} f(x)\;\mathcal{K}_\alpha(S,\mathrm{d}x)- \int_{\mathbb{R}^d}f(x)\;\mathcal{K}(S,\mathrm{d}x)\right| < \varepsilon.\]
\end{lemma}

\begin{proof}
	Fix~$f$,~$k$ and~$\varepsilon$ as in the statement of the lemma.  We choose constants as follows:
	\begin{itemize}
		\item since~$f$ is continuous with compact support, it is easy to see that we can choose~$\delta_0$ small enough that for any~$\delta \in (0,\delta_0)$ we have
			\[\sup_{\mu',\mu''} \left| \int_{\mathbb{R}^d} f\mathrm{d}\mu' - \int_{\mathbb{R}^d} f \mathrm{d}\mu'' \right| < \varepsilon,\]
			where the supremum is taken over all pairs of probability measures~$\mu',\mu''$ on Borel sets of~$\mathbb{R}^d$ with 
			\begin{equation}\label{eq_cond_mu}\sum_{Q \in \mathcal{C}_\delta}|\mu'(Q)-\mu''(Q)| < \delta.\end{equation}
			\item next, let~$\varepsilon' := \frac{\delta_0 \upupsilon_d r^d}{2}$, and choose~$\delta_1 = \delta_1(\varepsilon',k)$ as in Lemma~\ref{lem_helpful_kernel}.
	\end{itemize}  

	Letting~$\ell$ be large enough that the support of~$f$ is contained in~$B_\ell(o)$, assume that the event~$\mathrm{REG}_{\alpha}(\ell+rk+1,\delta_1)$ occurs, and let~$S \subseteq B_{rk}(o) \cap \mathcal{P}_{\alpha\lambda}$ be a set with at most~$k$ points. Using the triangle inequality and Lemma~\ref{lem_helpful_kernel}, we bound
	\begin{equation}\label{eq_main_tri_bound}\begin{split}
		&\sum_{Q \in \mathcal{C}_{\delta_1}} \left|\frac{\mathcal{N}_{\alpha}(S)}{\alpha\lambda} \cdot \mathcal{K}_{\alpha}(S,Q)- |S|\upupsilon_d r^d \cdot \mathcal{K}(S,Q)\right|\\
	&\le \sum_{x \in S}\sum_{Q \in \mathcal{C}_{\delta_1}} \left| \frac{|(\mathcal{P}_{\alpha\lambda}\cap Q \cap B_r(x) )\backslash S|}{\alpha\lambda} - \int_Q \mathds{1}\{\|y-x\| \le r\}\;\mathrm{d}y\right| \le  \varepsilon'.
	\end{split}\end{equation}
	Using the fact that~$\sum_Q \mathcal{K}_\alpha(S,Q) = \sum_Q \mathcal{K}(S,Q)=1$, this readily gives
	\begin{equation}\label{eq_sec_tri_bound}\begin{split}
		\left|\frac{\mathcal{N}_{\alpha}(S)}{\alpha\lambda} - |S|\upupsilon_dr^d\right|  
		&= \left| \frac{\mathcal{N}_\alpha(S)}{\alpha \lambda}\sum_{Q \in \mathcal{C}_\delta} \mathcal{K}_\alpha(S,Q) - |S|\upupsilon_dr^d \sum_{Q \in \mathcal{C}_\delta} \mathcal{K}(S,Q)\right|\\
		&\le\sum_{Q \in \mathcal{C}_{\delta_1}} \left|\frac{\mathcal{N}_{\alpha}(S)}{\alpha\lambda} \cdot \mathcal{K}_{\alpha}(S,Q)- |S|\upupsilon_d r^d \cdot \mathcal{K}(S,Q)\right| \le  \varepsilon'.
	\end{split}\end{equation}
	Next, the triangle inequality gives, for any~$Q \in \mathcal{C}_\delta$,
\begin{align*}
	|\mathcal{K}_{\alpha}(S,Q) - \mathcal{K}(S,Q)|\le\frac{1}{|S|\upupsilon_dr^d}&\left(\left|  |S|\upupsilon_dr^d- \frac{\mathcal{N}_{\alpha}(S)}{\alpha\lambda}\right|\cdot \mathcal{K}_{\alpha}(S,Q)\right.\\[.2cm]&\hspace{.1cm}
	\left.+ \left| \frac{\mathcal{N}_{\alpha}(S)}{\alpha\lambda} \cdot \mathcal{K}_{\alpha}(S,Q) - |S|\upupsilon_dr^d \cdot \mathcal{K}(S,Q)\right|  \right).
\end{align*}
	Combining this with~\eqref{eq_main_tri_bound} and~\eqref{eq_sec_tri_bound}, we obtain
\begin{align*}
	\sum_{Q \in \mathcal{C}_{\delta_1}} |\mathcal{K}_{\alpha}(S,Q) - \mathcal{K}(S,Q)| \le\frac{\varepsilon'}{|S|\upupsilon_dr^d}\left(  \sum_{Q \in \mathcal{C}_{\delta_1}} \mathcal{K}_{\alpha}(S,Q) + 1\right)\le \frac{2\varepsilon'}{\upupsilon_dr^d} = \delta_0.
\end{align*}
	This shows that~$\mathcal{K}_\alpha(S,\cdot)$ and~$\mathcal{K}(S,\cdot)$ are close enough in the sense that~\eqref{eq_cond_mu} is satisfied. The proof is now completed using~\eqref{eq_lim_of_reg}.
\end{proof}

The following is proved in a similar manner as Lemma~\ref{lem_prox_kernel_K}, only simpler, so we omit the details.
\begin{lemma}\label{lem_prox_kernel_LL}
	Let~$g:(0,\infty) \to \mathbb{R}$ be continuous with compact support,~$k\in \mathbb{N}$ and~$\varepsilon > 0$. There exists~$\alpha_0 > 0$ such that for any~$\alpha \ge \alpha_0$, we have that with probability larger than~$1-\varepsilon$,~$\mathcal{P}_{\alpha\lambda}$ satisfies the following. For any set~$S \subseteq B_{kr}(o) \cap \mathcal{P}_{\alpha\lambda}$ with~$|S|\le k$, we have
	\[\left|\int_{(0,\infty)} g(t)\;\mathcal{L}_\alpha(S,\mathrm{d}t)- \int_{(0,\infty)}g(t)\;\mathcal{L}(S,\mathrm{d}t)\right| < \varepsilon.\]
\end{lemma}

\section*{Acknowledgements}
We thank the team from Orange S.A.~as well as W.~K\"onig and A.~T\'obi\'as for inspiring discussions. This work was funded by the German Research Foundation under Germany's Excellence Strategy MATH+: The Berlin Mathematics Research Center, EXC-2046/1 project ID: 390685689, Orange Labs S.A., and the German Leibniz Association via the Leibniz Competition 2020.

This research was also supported by grants \#2017/10555-0, \#2019/19056-2, and \#2020/12868-9, S\~ao Paulo Research Foundation (FAPESP). The article was written when L.R.~de Lima was visiting the Department of Mathematics of the Bernoulli Institute at the University of Groningen. He is thankful for their hospitality. 

\bibliographystyle{abbrvnat}
\bibliography{references}

\end{document}